\documentclass[oneside,10pt]{amsart}
\usepackage{amsfonts}
\usepackage{amsthm}
\usepackage{amsmath}
\usepackage{amscd}
\usepackage[latin2]{inputenc}
\usepackage{t1enc}
\usepackage[mathscr]{eucal}
\usepackage{indentfirst}
\usepackage{graphicx}
\usepackage{graphics}
\usepackage{pict2e}
\usepackage{epic}
\numberwithin{equation}{section}
\usepackage[margin=2.9cm]{geometry}
\usepackage{epstopdf}
\newtheorem{theorem}[subsection]{Theorem}
\newtheorem{lemma}[subsection]{Lemma}

\theoremstyle{definition}

\newtheorem{definition}[subsection]{Definition}

\theoremstyle{property}

\makeatletter
\newcommand{\vast}{\bBigg@{2}}
\newcommand{\Vast}{\bBigg@{3}}
\makeatother

\begin{document}

\title{On new inequalities of Hermite-Hadamard type for functions whose fourth derivative absolute values are quasi-convex with applications}
\author{Imran Abbas Baloch, Basharat Rehman Ali}
\address{Imran Abbas Baloch\\Abdus Salam School of Mathematical Sciences\\
GC University, Lahore, Pakistan}\email{iabbasbaloch@gmail.com\\iabbasbaloch@sms.edu.pk}

\address{Basharat Rehman Ali\\Abdus Salam School of Mathematical Sciences\\
GC University, Lahore, Pakistan}\email{basharatrwp@gmail.com}



\begin{abstract} We establish some new inequalities of Hermite-Hadamard type for functions whose fourth derivatives absolute values are quasi-convex. Further, we give new identity.Using this new identity, we establish similar inequalities for left-hand side of Hermite-Hadamard result.Also, we present applications to special means.
\end{abstract}

\maketitle

\section{\bf{Introduction}}
A function $f:I \subseteq \mathbb{R} \rightarrow \mathbb{R}$  is called convex function if
$ f(\lambda x + (1 - \lambda)y) \leq \lambda f(x) + (1 - \lambda)f(y)  $
for all $x,y \in I$ and $\lambda \in [0,1].$ Geometrically, this means that if $P,Q$ and $R$ are three distinct points on graph of $f$ with $Q$ between $P$ and $R$, then $Q$ is on or below chord $PR$. There are many results associated with convex functions in the area of inequalities, but one of them is the classical Hermite-Hadamard inequalities:
$$
f\vast( \frac{a + b}{2} \vast) \leq \frac{1}{b - a} \int_{a}^{b} f(x) dx \leq \frac{f(a) + f(b)}{2},
$$
for all $a,b \in I$, with $a < b.$\\
Recently, numerous authors [1-7] developed and discussed Hermite-Hadamard's inequalities in terms of refinements, counter-parts, generalizations and new Hemitte-Hadamard's type inequalities.\\
The notion of quasi-convex function which is generalization of convex function is defined as:
\begin{definition}
A function $F:[a,b] \rightarrow \mathbb{R}$ is called quasi-convex on $[a,b]$, if
$$
 f(\lambda x + (1 - \lambda)y) \leq \max\{f(x),f(y)\} ,\;\; \forall x,y \in [a,b].
$$
\end{definition}
Any convex function is quasi-convex but converse is not true in general(See for example \cite{AD1}). D.A Ion \cite{DAI} established inequalities of right hand side of Hermite-Hadamard's type inequality for functions whose derivatives in absolute values are quasi-convex functions. These inequalities appear in the following theorems:
\begin{theorem}
Let $f: I^{\circ} \subseteq \mathbb{R} \rightarrow \mathbb{R}$ be a function, differentiable on $I^{\circ}$ with $a,b \in I^{\circ}$ and $a < b$. If $|f'|$ is quasi-convex on $[a,b]$, then we have:

$$ \vast|  \frac{f(a) + f(b)}{2} - \frac{1}{b - a}\int_{a}^{b} f(x)dx \vast| \leq \frac{(b - a)}{4} \max\{|f'(a)|,|f'(b)|\}.$$
\end{theorem}
\begin{theorem}
Let $f: I^{\circ} \subseteq \mathbb{R} \rightarrow \mathbb{R}$ be a function, differentiable on $I^{\circ}$ with $a,b \in I^{\circ}$ and $a < b$. If $|f'|^{\frac{p}{p - 1}}$ is quasi-convex on $[a,b]$, then we have:
$$ \vast|  \frac{f(a) + f(b)}{2} - \frac{1}{b - a}\int_{a}^{b} f(x) dx\vast| \leq \frac{(b - a)}{2(p + 1)^{\frac{1}{p}}}( \max\{|f'(a)|^{\frac{p}{p - 1}},|f'(b)|^{\frac{p}{p - 1}}\})^{\frac{p - 1}{p}}.$$
\end{theorem}
\begin{theorem}
Let $f: I^{\circ} \subseteq \mathbb{R} \rightarrow \mathbb{R}$ be a function, twice differentiable on $I^{\circ}$ with $a,b \in I^{\circ}$ and $a < b$. If $|f''|$ is quasi-convex on $[a,b]$, then we have:

$$ \vast|  \frac{f(a) + f(b)}{2} - \frac{1}{b - a}\int_{a}^{b}f(x) dx \vast| \leq \frac{(b - a)^{2}}{12} \max\{|f''(a)|,|f''(b)|\}.$$
\end{theorem}
In paper \cite{SSC}, S.Qaisar, S.Hussain, C. He established new refined inequalities of right hand side of Hermite-Hadamard result for the class of functions whose third derivatives at certain powers are quasi-convex functions as follow:
\begin{theorem}
Let $f: I \subseteq \mathbb{R} \rightarrow \mathbb{R}$ be thrice differentiable mapping on $I^{\circ}$ such that $f''' \in L[a,b]$, where $a,b \in I$ with $a < b$. If $|f'''|$ is quasi-convex on $[a,b]$, then we have the following inequality:
$$ \vast|  \frac{f(a) + f(b)}{2} - \frac{1}{b - a}\int_{a}^{b}f(x) dx - \frac{b - a}{12}[f'(b) - f'(a)]  \vast| \leq \frac{(b - a)^{3}}{192} \max\{|f'''(a)|,|f'''(b)|\}.$$
\end{theorem}
\begin{theorem}
Let $f: I \subseteq \mathbb{R} \rightarrow \mathbb{R}$ be three time differentiable mapping on $I^{\circ}$ such that $f''' \in L[a,b]$, where $a,b \in I$ with $a < b$. If $|f'''|^{\frac{p}{p - 1}}$ is quasi-convex on $[a,b]$, and $p > 1$, then we have the following inequality:
$$ \vast|  \frac{f(a) + f(b)}{2} - \frac{1}{b - a}\int_{a}^{b}f(x)dx - \frac{b - a}{12}[f'(b) - f'(a)]  \vast| \leq \frac{(b - a)^{3}}{96} \vast( \frac{1}{p + 1} \vast)^{\frac{1}{p}} (\max\{|f'''(a)|^{q},|f'''(b)|^{q}\})^{\frac{1}{q}}.$$
\end{theorem}
\begin{theorem}
Let $f: I \subseteq \mathbb{R} \rightarrow \mathbb{R}$ be thrice differentiable mapping on $I^{\circ}$ such that $f''' \in L[a,b]$, where $a,b \in I$ with $a < b$. If $|f'''|^{q}$ is quasi-convex on $[a,b]$, and $q \geq 1$, then we have following inequality:
$$ \vast|  \frac{f(a) + f(b)}{2} - \frac{1}{b - a}\int_{a}^{b}f(x) dx - \frac{b - a}{12}[f'(b) - f'(a)]  \vast| \leq \frac{(b - a)^{3}}{192} (\max\{|f'''(a)|^{q},|f'''(b)|^{q}\})^{\frac{1}{q}}.$$
\end{theorem}
In this paper, we establish new refined inequalities of the right hand side of Hermite-Hadamard result for the class of functions whose fourth derivative at certain powers are quasi-convex functions. Further, we establish new identity using which, we establish new refined inequalities of left hand side of Hermit-Hadamard result for the same class of functions considered earlier.

\section{\textbf{Main Results}}
For establishing new inequalities of right hand side of Hermite-Hadamard result for the functions whose fourth derivative at certain powers are quasi-convex, we need the following identity:
\begin{lemma}\label{ML1}
Let Let $f: I \subseteq \mathbb{R} \rightarrow \mathbb{R}$ be four times differentiable mapping on $I^{\circ}$ such that $f^{(iv)} \in L[a,b]$, where $a,b \in I$ with $a < b$, then
$$ \frac{1}{b - a}\int_{a}^{b}f(x) dx  +   \frac{b - a}{12}[f'(b) - f'(a)] - \frac{f(a) + f(b)}{2}  =  \frac{(b - a)^{4}}{24} \int_{0}^{1} (\lambda(1 - \lambda))^{2} f^{(iv)} (a\lambda + (1 - \lambda)b) d\lambda$$
\end{lemma}
\begin{theorem}
Let $f: I \subseteq \mathbb{R} \rightarrow \mathbb{R}$ be a four times differentiable mapping on $I^{\circ}$ such that $f^{(iv)} \in L[a,b]$, where $a,b \in I$ with $a < b$. If $|f^{iv}|$ is quasi-convex on $[a,b]$, then we have the following inequality:
\begin{equation}\label{ME1}
\vast|  \frac{f(a) + f(b)}{2} - \frac{1}{b - a}\int_{a}^{b}f(x)dx - \frac{b - a}{12}[f'(b) - f'(a)]  \vast| \leq \frac{(b - a)^{4}}{720} \max\{|f^{(iv)}(a)|,|f^{(iv)}(b)|\}.
\end{equation}
\end{theorem}
\begin{proof}
Using Lemma \ref{ML1} and quasi-convexity of $|f^{(iv)}|$, we get

\begin{eqnarray*}
\vast|  \frac{f(a) + f(b)}{2} - \frac{1}{b - a}\int_{a}^{b}f(x)dx - \frac{b - a}{12}[f'(b) - f'(a)]  \vast| &\leq& \frac{(b - a)^{4}}{24} \int_{0}^{1} (\lambda(1 - \lambda))^{2} \big|f^{(iv)} (a\lambda + (1 - \lambda)b)\big| d\lambda\\
&\leq& \frac{(b - a)^{4}}{24} \max\{ |f^{(iv)}(a)|,|f^{(iv)}(a)|\}\int_{0}^{1} (\lambda(1 - \lambda))^{2}d\lambda\\
&=&  \frac{(b - a)^{4}}{720} \max\{ |f^{(iv)}(a)|,|f^{(iv)}(a)|\}
\end{eqnarray*}
the proof is completed.
\end{proof}
\begin{theorem}
Let $f: I \subseteq \mathbb{R} \rightarrow \mathbb{R}$ be a four times differentiable mapping on $I^{\circ}$ such that $f^{(iv)} \in L[a,b]$, where $a,b \in I$ with $a < b$. If $|f^{(iv)}|^{\frac{p}{p - 1}}$ is quasi-convex on $[a,b]$, and $p > 1$, then we have the following inequality:
 \begin{equation}\label{ME2}
 \vast|  \frac{f(a) + f(b)}{2} - \frac{1}{b - a}\int_{a}^{b}f(x)dx - \frac{b - a}{12}[f'(b) - f'(a)]  \vast| \leq \frac{(b - a)^{4}}{24} \beta^{\frac{1}{p}}(2p +1, 2p + 1) (\max\{|f^{(iv)}(a)|^{q},|f^{(iv)}(b)|^{q}\})^{\frac{1}{q}},
 \end{equation}
where $q = \frac{p}{p - 1}.  $
\end{theorem}
\begin{proof}
Using Lemma \ref{ML1}, Holder's inequality and quasi-convexity of $|f^{(iv)}|^{\frac{p}{p - 1}}$, we get
$$
\vast|  \frac{f(a) + f(b)}{2} - \frac{1}{b - a}\int_{a}^{b}f(x)dx - \frac{b - a}{12}[f'(b) - f'(a)]  \vast| $$
$$\leq \frac{(b - a)^{4}}{24} \int_{0}^{1} (\lambda(1 - \lambda))^{2} \big|f^{(iv)} (a\lambda + (1 - \lambda)b)\big| d\lambda$$
$$\leq \frac{(b - a)^{4}}{24}\vast(\int_{0}^{1} (\lambda(1 - \lambda))^{2p}d\lambda\vast)^{\frac{1}{p}} \vast(\int_{0}^{1}  \big|f^{(iv)} (a\lambda + (1 - \lambda)b)|^{q}\vast)^{\frac{1}{q}}$$
$$\leq \frac{(b - a)^{4}}{24}\vast(\int_{0}^{1} (\lambda(1 - \lambda))^{2p}d\lambda\vast)^{\frac{1}{p}} (\max\{|f^{(iv)}(a)|^{q},|f^{(iv)}(b)|^{q}\})^{\frac{1}{q}}$$
It is easy to note that
$$\beta(2p + 1,2p + 1)    =  \int_{0}^{1} (\lambda(1 - \lambda))^{2p}d\lambda $$
which completes the proof.
\end{proof}
\begin{theorem}
Let $f: I \subseteq \mathbb{R} \rightarrow \mathbb{R}$ be a four times differentiable mapping on $I^{\circ}$ such that $f^{(iv)} \in L[a,b]$, where $a,b \in I$ with $a < b$. If $|f^{{iv}}|^{q}$ is quasi-convex on $[a,b]$, and $q \geq 1$, then we have following inequality:
 \begin{equation}\label{ME3}
 \vast|  \frac{f(a) + f(b)}{2} - \frac{1}{b - a}\int_{a}^{b}f(x)dx - \frac{b - a}{12}[f'(b) - f'(a)]  \vast| \leq \frac{(b - a)^{4}}{720}  (\max\{|f^{(iv)}(a)|^{q},|f^{((iv))}(b)|^{q}\})^{\frac{1}{q}}.
 \end{equation}
\end{theorem}
\begin{proof}
Using Lemma \ref{ML1}, power mean inequality and quasi-convexity of $|f^{(iv)}|^{q}$, we get
$$
\vast|  \frac{f(a) + f(b)}{2} - \frac{1}{b - a}\int_{a}^{b}f(x)dx - \frac{b - a}{12}[f'(b) - f'(a)]  \vast| $$
$$\leq \frac{(b - a)^{4}}{24} \int_{0}^{1} (\lambda(1 - \lambda))^{2} \big|f^{(iv)} (a\lambda + (1 - \lambda)b)\big| d\lambda$$
$$\leq \frac{(b - a)^{4}}{24}\vast(\int_{0}^{1} (\lambda(1 - \lambda))^{2}d\lambda\vast)^{1 - \frac{1}{q}} \vast(\int_{0}^{1}(\lambda(1 - \lambda))^{2}  \big|f^{(iv)} (a\lambda + (1 - \lambda)b)|^{q}\vast)^{\frac{1}{q}}$$
$$\leq \frac{(b - a)^{4}}{24}\vast(\frac{1}{30}\vast)^{1 - \frac{1}{q}} \vast(\frac{1}{30}\max\{|f^{(iv)}(a)|^{q},|f^{(iv)}(b)|^{q}\}\vast)^{\frac{1}{q}}$$
$$= \frac{(b - a)^{4}}{720}  (\max\{|f^{(iv)}(a)|^{q},|f^{(iv)}(b)|^{q}\})^{\frac{1}{q}}.$$
which completes the proof.
\end{proof}
Now, to develop new refined inequalities of left hand side of Hermite-Hadamard result for the class of functions whose third derivatives at certain powers are quasi-convex, we need the following identity:
\begin{lemma}\label{ML2}
$$f\vast( \frac{a + b}{2} \vast) - \frac{1}{b - a} \int_{a}^{b} f(x) dx + \frac{b - a}{24}[f'(b) - f'(a)]$$
 $$= \frac{(b - a)^{3}}{24} \vast[ \int_{0}^{\frac{1}{2}} \lambda ( 1 - 2\lambda) ( 1 + 2\lambda) f'''(\lambda a + (1 - \lambda)b)d\lambda - \int_{0}^{\frac{1}{2}} \lambda ( 1 - 2\lambda) ( 1 + 2\lambda) f'''(\lambda b + (1 - \lambda)a)d\lambda \vast] $$
\end{lemma}
\begin{proof}
Integrating by parts, we have
$$
\int_{0}^{\frac{1}{2}} \lambda ( 1 - 2\lambda) ( 1 + 2\lambda) f'''(\lambda a + (1 - \lambda)b)d\lambda$$
 $$= \frac{1}{b - a}\int_{0}^{\frac{1}{2}} ( 1 - 12 \lambda^{2}) f''(\lambda a + (1 - \lambda) b)d\lambda$$
$$= \frac{2}{(b - a)^{2}}f'\vast( \frac{a + b}{2} \vast) + \frac{f'(b)}{(b - a)^{2}} - \frac{24}{(b - a)^{2}} \int_{0}^{\frac{1}{2}} \lambda f'(\lambda a + (1 - \lambda) b)d\lambda$$
$$
= \frac{2}{(b - a)^{2}}f'\vast( \frac{a + b}{2} \vast) + \frac{f'(b)}{(b - a)^{2}} + \frac{12}{(b - a)^{3}}f\vast( \frac{a + b}{2} \vast) + \frac{24}{(a - b)^{4}} \int_{b}^{\frac{a + b}{2}} f(x)dx
$$
and
$$
\int_{0}^{\frac{1}{2}} \lambda ( 1 - 2\lambda) ( 1 + 2\lambda) f'''(\lambda b + (1 - \lambda)a)d\lambda
$$
$$= \frac{-1}{b - a}\int_{0}^{\frac{1}{2}} ( 1 - 12 \lambda^{2}) f''(\lambda b + (1 - \lambda) a)d\lambda$$
$$= \frac{2}{(b - a)^{2}}f'\vast( \frac{a + b}{2} \vast) + \frac{f'(a)}{(b - a)^{2}} - \frac{24}{(b - a)^{2}} \int_{0}^{\frac{1}{2}} \lambda f'(\lambda b + (1 - \lambda) a)d\lambda$$
$$= \frac{24}{(b - a)^{2}}f'\vast( \frac{a + b}{2} \vast) + \frac{f'(a)}{(b - a)^{2}} - \frac{12}{(b - a)^{3}}f\vast( \frac{a + b}{2} \vast) + \frac{2}{(b - a)^{4}} \int_{a}^{\frac{a + b}{2}} f(x)dx
$$
this ends the proof.
\end{proof}

\begin{theorem}
Let $f: I \subseteq \mathbb{R} \rightarrow \mathbb{R}$ be a three time differentiable mapping on $I^{\circ}$ such that $f''' \in L[a,b]$, where $a,b \in I$ with $a < b$. If $|f'''|$ is quasi-convex on $[a,b]$, then we have following inequality:
 \begin{equation}\label{ME4}
 \vast|  f\vast( \frac{a + b}{2} \vast) - \frac{1}{b - a} \int_{a}^{b} f(x) dx + \frac{b - a}{24}[f'(b) - f'(a)]  \vast| \leq \frac{(b - a)^{3}}{192} \vast( \max\{|f'''(a)|,|f'''(b)|  \}\vast).
 \end{equation}
\end{theorem}
\begin{proof}
Using Lemma \ref{ML2} and quasi-convexity of $|f'''|$, we get
$$ \vast|  f\vast( \frac{a + b}{2} \vast) - \frac{1}{b - a} \int_{a}^{b} f(x) dx + \frac{b - a}{24}[f'(b) - f'(a)]  \vast|$$
$$\leq \frac{(b - a)^{3}}{24} \vast[ \int_{0}^{\frac{1}{2}} \lambda ( 1 - 2\lambda) ( 1 + 2\lambda) |f'''(\lambda a + (1 - \lambda)b)|d\lambda + \int_{0}^{\frac{1}{2}} \lambda ( 1 - 2\lambda) ( 1 + 2\lambda) |f'''(\lambda b + (1 - \lambda)a)|d\lambda \vast] $$
$$\leq \frac{(b - a)^{3}}{24}\vast( \max\{|f'''(a)|,|f'''(b)|  \}\vast) \vast[ \int_{0}^{\frac{1}{2}} \lambda ( 1 - 2\lambda) ( 1 + 2\lambda) d\lambda + \int_{0}^{\frac{1}{2}} \lambda ( 1 - 2\lambda) ( 1 + 2\lambda) d\lambda \vast] $$
$$= \frac{(b - a)^{3}}{192}\vast( \max\{|f'''(a)|,|f'''(b)|  \}\vast)$$
this complete the proof.
\end{proof}
\begin{theorem}
Let $f: I \subseteq \mathbb{R} \rightarrow \mathbb{R}$ be a three time differentiable mapping on $I^{\circ}$ such that $f''' \in L[a,b]$, where $a,b \in I$ with $a < b$. If $|f^{'''}|^{\frac{p}{p - 1}}$ is quasi-convex on $[a,b]$, and $p > 1$, then we have following inequality:
\begin{equation}\label{ME5}
 \vast| f\vast( \frac{a + b}{2} \vast) - \frac{1}{b - a} \int_{a}^{b} f(x) dx + \frac{b - a}{24}[f'(b) - f'(a)]  \vast|
\end{equation}
 \begin{equation*}
 \leq \frac{(b - a)^{3}}{96} (\frac{1}{p + 1})^{\frac{1}{p}}  \big(\max\{|f'''(a)|^{q},|f'''(b)|^{q}\}\big)^{\frac{1}{q}},
 \end{equation*}
where $q = \frac{p}{p - 1}.  $
\end{theorem}
\begin{proof}
Using Lemma \ref{ML2}, Holder's inequality and quasi-convexity of $|f^{'''}|^{\frac{p}{p - 1}}$, we get
$$ \vast|  f\vast( \frac{a + b}{2} \vast) - \frac{1}{b - a} \int_{a}^{b} f(x) dx + \frac{b - a}{24}[f'(b) - f'(a)]  \vast|$$
$$\leq \frac{(b - a)^{3}}{24} \vast[ \int_{0}^{\frac{1}{2}} \lambda ( 1 - 2\lambda) ( 1 + 2\lambda) |f'''(\lambda a + (1 - \lambda)b)|d\lambda + \int_{0}^{\frac{1}{2}} \lambda ( 1 - 2\lambda) ( 1 + 2\lambda) |f'''(\lambda b + (1 - \lambda)a)|d\lambda \vast] $$
$$\leq \frac{(b - a)^{3}}{24} \vast[ \vast(\int_{0}^{\frac{1}{2}} \lambda ( 1 - 2\lambda)^{p} ( 1 + 2\lambda)^{p}d \lambda\vast)^{\frac{1}{p}} \vast(\int_{0}^{\frac{1}{2}} \lambda|f'''(\lambda a + (1 - \lambda)b)|^{q}d\lambda\vast)^{\frac{1}{q}}$$
 $$+ \vast( \int_{0}^{\frac{1}{2}} \lambda ( 1 - 2\lambda)^{p} ( 1 + 2\lambda)^{p}d\lambda\vast)^{\frac{1}{p}} \vast( \int_{0}^{\frac{1}{2}} \lambda|f'''(\lambda b + (1 - \lambda)a)|^{q}d\lambda \vast)^{\frac{1}{q}} \vast] $$

$$\leq \frac{(b - a)^{3}}{24}\vast( \max\{|f'''(a)|^{q},|f'''(b)|^{q}  \}\vast)^{\frac{1}{q}} \vast[ \vast(\int_{0}^{\frac{1}{2}} \lambda ( 1 - 2\lambda)^{p} ( 1 + 2\lambda)^{p}d \lambda\vast)^{\frac{1}{p}} \vast(\int_{0}^{\frac{1}{2}} \lambda d\lambda\vast)^{\frac{1}{q}}$$
 $$+ \vast( \int_{0}^{\frac{1}{2}} \lambda ( 1 - 2\lambda)^{p} ( 1 + 2\lambda)^{p}d\lambda\vast)^{\frac{1}{p}} \vast( \int_{0}^{\frac{1}{2}}\lambda d\lambda \vast)^{\frac{1}{q}} \vast] $$
$$
= \frac{(b - a)^{3}}{96} \vast(\frac{1}{p + 1} \vast)^{\frac{1}{p}}\vast( \max\{|f'''(a)|^{q},|f'''(b)|^{q}  \}\vast)^{\frac{1}{q}}
$$
\end{proof}

\begin{theorem}
Let $f: I \subseteq \mathbb{R} \rightarrow \mathbb{R}$ be a three time differentiable mapping on $I^{\circ}$ such that $f''' \in L[a,b]$, where $a,b \in I$ with $a < b$. If $|f^{'''}|^{q}$ is quasi-convex on $[a,b]$, and $q \geq 1$, then we have following inequality:
\begin{equation}\label{ME6}
\vast| f\vast( \frac{a + b}{2} \vast) - \frac{1}{b - a} \int_{a}^{b} f(x) dx + \frac{b - a}{24}[f'(b) - f'(a)]  \vast|
\end{equation}
 $$\leq \frac{(b - a)^{3}}{192}   \big(\max\{|f'''(a)|^{q},|f'''(b)|^{q}\}\big)^{\frac{1}{q}}.$$
\end{theorem}
\begin{proof}
Using Lemma \ref{ML2}, power mean inequality and quasi-convexity of $|f'''|^{q}$, we get
$$ \vast|  f\vast( \frac{a + b}{2} \vast) - \frac{1}{b - a} \int_{a}^{b} f(x) dx + \frac{b - a}{24}[f'(b) - f'(a)]  \vast|$$
$$\leq \frac{(b - a)^{3}}{24} \vast[ \int_{0}^{\frac{1}{2}} \lambda ( 1 - 2\lambda) ( 1 + 2\lambda) |f'''(\lambda a + (1 - \lambda)b)|d\lambda + \int_{0}^{\frac{1}{2}} \lambda ( 1 - 2\lambda) ( 1 + 2\lambda) |f'''(\lambda b + (1 - \lambda)a)|d\lambda \vast] $$
$$\leq \frac{(b - a)^{3}}{24} \vast[ \vast(\int_{0}^{\frac{1}{2}} \lambda ( 1 - 2\lambda) ( 1 + 2\lambda)d\lambda\vast)^{1 - \frac{1}{q}} \vast(\int_{0}^{\frac{1}{2}} \lambda ( 1 - 2\lambda) ( 1 + 2\lambda) |f'''(\lambda a + (1 - \lambda)b)|^{q}d\lambda| \vast)^{\frac{1}{q}} $$
$$+ \vast(\int_{0}^{\frac{1}{2}} \lambda ( 1 - 2\lambda) ( 1 + 2\lambda)d\lambda\vast)^{1 - \frac{1}{q}} \vast(\int_{0}^{\frac{1}{2}} \lambda ( 1 - 2\lambda) ( 1 + 2\lambda)|f'''(\lambda b + (1 - \lambda)a)|^{q}d\lambda\vast)^{\frac{1}{q}} \vast] $$

$$\leq \frac{(b - a)^{3}}{24}\vast( \max\{|f'''(a)|^{q},|f'''(b)|^{q}  \}\vast)^{\frac{1}{q}} \vast[ \vast(\int_{0}^{\frac{1}{2}} \lambda ( 1 - 2\lambda) ( 1 + 2\lambda)d\lambda\vast)^{1 - \frac{1}{q}} \vast(\int_{0}^{\frac{1}{2}} \lambda ( 1 - 2\lambda) ( 1 + 2\lambda)d\lambda| \vast)^{\frac{1}{q}} $$
$$+ \vast(\int_{0}^{\frac{1}{2}} \lambda ( 1 - 2\lambda) ( 1 + 2\lambda)d\lambda\vast)^{1 - \frac{1}{q}} \vast(\int_{0}^{\frac{1}{2}} \lambda ( 1 - 2\lambda) ( 1 + 2\lambda)d\lambda\vast)^{\frac{1}{q}} \vast] $$
$$\leq \frac{(b - a)^{3}}{192}\vast( \max\{|f'''(a)|^{q},|f'''(b)|^{q}  \}\vast)^{\frac{1}{q}} $$
\end{proof}
the proof is so completed.
\section{\bf{Application to some special means}}
We now consider the application of our theorem to the special means.\\
For positive numbers $a > 0$ and $b > 0$, define $A(a,b) = \frac{a + b}{2} $   and
$$
L{_p}(a,b) =  \left\{
                                                            \begin{array}{ll}
                                                              \vast[ \frac{b^{p + 1} - a^{p + 1}}{(p + 1)(b - a)} \vast], & \hbox{$ p \neq -1, 0$} \\
                                                              {}\\
                                                              \frac{b - a}{\ln b - \ln a }, & \hbox{$p = -1$}\\
                                                              {}\\
                                                              \frac{1}{e} \vast( \frac{b^{b}}{a^{a}}\vast)^\frac{1}{b - a}, & \hbox{$p = 0$}
                                                            \end{array}
                                                          \right.    $$
We know that $A$ and $L_{p}$ respectively are called the arithmetic and generalized logarithmic means of two positive numbers $a$ and $b$. By applying Hermite-Hadamard type inequalities established in Section 2, we are in a position to construct some inequalities for special means $A$ and $L_{P}$. Consider the following function:
\begin{equation}\label{F1}
f(x) = \frac{x^{ \alpha + 4}}{(\alpha + 1)(\alpha + 2)(\alpha + 3)(\alpha + 4)}
\end{equation}
for $ 0 < \alpha \leq 1$ and $ x > 0$. Since $f^{(iv)}(x) = x^{\alpha}$ and $ (\lambda x + (1 - \lambda)y)^{\alpha} \leq \lambda^{\alpha} x^{\alpha} + (1 - \lambda)^{\alpha} y^{\alpha}$ for all $x,y > 0$ and $\lambda \in [0,1]$, then $f^{(iv)}(x) = x^{\alpha}$ is $\alpha$-convex function on $\mathbb{R}^{+}$ and
$$
\frac{f(a) + f(b)}{2} = \frac{1}{(\alpha + 1)(\alpha + 2)(\alpha + 3)(\alpha + 4)} A(a^{\alpha + 4},b^{\alpha + 4}),
$$
$$
\frac{1}{b - a} \int_{a}^{b} f(x) dx = \frac{1}{(\alpha + 1)(\alpha + 2)(\alpha + 3)(\alpha + 4)} L_{\alpha + 4}(a,b),
$$
$$
f'(b) - f'(a) = \frac{b - a}{(\alpha + 1)(\alpha + 2)} L_{\alpha + 2}(a ,b)
$$
\begin{theorem}
For positive numbers $a$ and $b$ such that $ b > a$ and $ 0 < \alpha \leq 1$, we have
$$ \vast|12A(a^{\alpha + 4},b^{\alpha + 4}) -  12L_{\alpha + 4}(a,b) - (b - a)^{2} (\alpha + 3)(\alpha + 4)(\alpha + 4) L_{\alpha + 2}(a,b) \vast|$$
$$
\leq \frac{(b -a)^{4}}{60} (\alpha + 1)(\alpha + 2)(\alpha + 3)(\alpha + 4) \max\{|a^{\alpha}|,|b^{\alpha}|   \}
$$
\end{theorem}
\begin{proof}
The assertion follows from inequality (\ref{ME1} ) applied to mapping (\ref{F1}).
\end{proof}

\begin{theorem}
For positive numbers $a$ and $b$ such that $ b > a$ and $ 0 < \alpha \leq 1$, we have
$$ \vast|12A(a^{\alpha + 4},b^{\alpha + 4}) -  12L_{\alpha + 4}(a,b) - (b - a)^{2} (\alpha + 3)(\alpha + 4)(\alpha + 4) L_{\alpha + 2}(a,b) \vast|$$
$$
\leq \frac{(b -a)^{4}}{2}\beta^{\frac{1}{p}}(2p + 1,2p + 1) (\alpha + 1)(\alpha + 2)(\alpha + 3)(\alpha + 4) (\max\{|a^{\alpha}|^{q},|b^{\alpha}|^{q}   \})^{\frac{1}{q}}.
$$
\end{theorem}
\begin{proof}
The assertion follows from inequality (\ref{ME2}) applied to the mapping (\ref{F1}).
\end{proof}
\begin{theorem}
For positive numbers $a$ and $b$ such that $ b > a$ and $ 0 < \alpha \leq 1$, we have
$$ \vast|12A(a^{\alpha + 4},b^{\alpha + 4}) -  12L_{\alpha + 4}(a,b) - (b - a)^{2} (\alpha + 3)(\alpha + 4)(\alpha + 4) L_{\alpha + 2}(a,b) \vast|$$
$$
\leq \frac{(b -a)^{4}}{60} (\alpha + 1)(\alpha + 2)(\alpha + 3)(\alpha + 4) (\max\{|a^{\alpha}|^{q},|b^{\alpha}|^{q}   \})^{\frac{1}{q}}.
$$

\end{theorem}
\begin{proof}
The assertion follows from inequality (\ref{ME3}) applied to the mapping (\ref{F1}).

\end{proof}

\begin{theorem}
For positive numbers $a$ and $b$ such that $ b > a$ and $ 0 < \alpha \leq 1$, we have
$$ \vast|12A^{\alpha + 4}(a^{\alpha + 4},b^{\alpha + 4}) -  12L_{\alpha + 4}(a,b) - (b - a)^{2} (\alpha + 3)(\alpha + 4)(\alpha + 4) L_{\alpha + 2}(a,b) \vast|$$
$$
\leq \frac{(b -a)^{3}}{16} (\alpha + 1)(\alpha + 2)(\alpha + 3)(\alpha + 4) (\max\{|a^{\alpha}|,|b^{\alpha}|   \}).
$$

\end{theorem}
\begin{proof}
The assertion follows from inequality (\ref{ME4}) applied to the mapping (\ref{F1}).
\end{proof}
\begin{theorem}
For positive numbers $a$ and $b$ such that $ b > a$ and $ 0 < \alpha \leq 1$, we have
$$ \vast|12A^{\alpha + 4}(a^{\alpha + 4},b^{\alpha + 4}) -  12L_{\alpha + 4}(a,b) - (b - a)^{2} (\alpha + 3)(\alpha + 4)(\alpha + 4) L_{\alpha + 2}(a,b) \vast|$$
$$\leq \frac{(b - a)^{3}}{8} (\frac{1}{p + 1})^{\frac{1}{p}} (\alpha + 1)(\alpha + 2)(\alpha + 3)(\alpha + 4)  \big(\max\{|a^{\alpha}|^{q},|b^{\alpha}|^{q}\}\big)^{\frac{1}{q}}$$
\end{theorem}
\begin{proof}
The assertion follows from inequality (\ref{ME5}) applied to the mapping (\ref{F1}).
\end{proof}
\begin{theorem}
For positive numbers $a$ and $b$ such that $ b > a$ and $ 0 < \alpha \leq 1$, we have
$$ \vast|12A^{\alpha + 4}(a^{\alpha + 4},b^{\alpha + 4}) -  12L_{\alpha + 4}(a,b) - (b - a)^{2} (\alpha + 3)(\alpha + 4)(\alpha + 4) L_{\alpha + 2}(a,b) \vast|$$
$$\leq \frac{(b - a)^{3}}{16} (\alpha + 1)(\alpha + 2)(\alpha + 3)(\alpha + 4)  \big(\max\{|a^{\alpha}|^{q},|b^{\alpha}|^{q}\}\big)^{\frac{1}{q}}.$$
\end{theorem}
\begin{proof}
The assertion follows from inequality (\ref{ME6}) applied to the mapping (\ref{F1}).
\end{proof}

\end{document}